\documentclass[11pt]{article}
\usepackage{amsfonts}
\usepackage{amssymb}
\usepackage{graphics}
\usepackage{amsmath,epsfig,psfrag}
\usepackage[english]{babel}
\usepackage[all]{xy}
\usepackage{amscd}
\usepackage{amsthm}
\usepackage{latexsym}
\usepackage{graphicx}
\usepackage{mathrsfs}

\def\Z{{\Bbb Z}}

\def\a{{\alpha}}
\def\R{{\mathbb R}}

\newtheorem{theorem}{Theorem}[section]

\newtheorem{lemma}[theorem]{Lemma}
\newtheorem{proposition}[theorem]{Proposition}

\theoremstyle{definition}
\newtheorem{example}[theorem]{Example}
\newtheorem{definition}[theorem]{Definition}

\theoremstyle{remark}
\newtheorem{remark}[theorem]{Remark}

\newcounter{fig}
\begin{document}

\title{Embedding compact surfaces into the 3-dimensional Euclidean space with maximum symmetry}
\author{Chao Wang, Shicheng Wang, Yimu Zhang, Bruno Zimmermann}
\date{Dedicated to Professor Boju Jiang on his 80th birthday}
\maketitle

\begin{abstract}
The symmetries of surfaces which can be embedded into the symmetries of the 3-dimensional Euclidean space  $\R^3$ are
easier to feel by human's intuition.

We give the maximum order of finite group actions on $(\R^3, \Sigma)$ among all possible embedded closed/bordered surfaces with given geometric/algebraic genus $>1$ in $\R^3$. We also identify the topological types of the bordered surfaces realizing the maximum order, and find simple representative embeddings for such surfaces.
\end{abstract}


\section{Introduction}
The maximum orders of finite group actions on surfaces have been studied for a long time, and a rather recent topic is to study the maximum orders of those finite group actions on closed surfaces which can extend over a given compact 3-manifold. Let $\Sigma_g$ denote the closed orientable surface of (geometric) genus $g$. For each compact surface $\Sigma$, let $\a(\Sigma)$ denote its {\it algebraic genus}, defined as the rank of the fundamental group $\pi_1(\Sigma)$. In the following, we discuss some sample results about these maximum order problems:

(1) Maximum orders of finite group actions on surfaces: (i) A classical result of Hurwitz states that the maximum order of orientation-preserving finite group actions on $\Sigma_g$ with $g>1$ is at most $84(g-1)$, \cite{Hu} in 1893. (ii) The maximum order of finite cyclic group actions on $\Sigma_g$ with $g>1$ is $4g+4$ for even $g$ and $4g+2$ for odd $g$, \cite{St} in 1935. (iii) The maximum order of finite group actions on bordered surfaces of algebraic genus $\a>1$ is at most $12(\a-1)$, \cite{Ma} in 1975. In (i) and (iii), to determine those maximum orders for concrete genera are still hard questions in general, and there are numerous interesting partial results.

(2) Maximum orders of finite group actions on surfaces which can extend over a given 3-manifold $M$: (i) The maximum order of finite group actions on $\Sigma_g$ with $g>1$ is at most $12(g-1)$ when $M$ is a handlebody and $\Sigma_g=\partial M$, \cite{Zi} in 1979. (ii) Much more recently the maximum order of extendable finite group actions on $\Sigma_g$ is determined when $M$ is the 3-sphere $S^3$,
for cyclic group, see \cite{WWZZ1} for orientation-preserving case and \cite{WZh} for general case;
for general finite groups, see \cite{WWZZ2} for orientation-preserving case and \cite{WWZ1} for general case. (iii) Some progress has been made when $M$ is the 3-torus $T^3$, see \cite{WWZ2} for cyclic case and \cite{BRWW} for general case. In (i) and (iii), to determine those maximum orders for concrete genera are still not solved.

Surfaces belong to the most familiar topological subjects mostly because they can be seen staying in the 3-dimensional Euclidean space $\R^3$ in various manners. The symmetries of surfaces which can be embedded into the symmetries of $\R^3$ are easier to feel by intuition. Hence it will be more natural to wonder the maximum orders of finite group actions on surfaces which extend over the 3-dimensional Euclidean space $\R^3$. In this paper, we will study this maximum order problem for all compact (closed/bordered) surfaces with given (geometric/algebraic) genera.

To state our results, we need some notions and definitions. We always assume that the manifolds, embeddings and group actions are smooth, and the group actions on $\R^3$ are faithful. Let $O(3)$ denote the isometry group of the unit sphere in the 3-dimensional Euclidean space $\R^3$, and let $SO(3)$ denote the orientation-preserving isometry group of the unit sphere in $\R^3$. It is known that any  finite group $G$ acting on $\R^3$ can be conjugated into $O(3)$, especially it can be conjugated into $SO(3)$ if the $G$ action is orientation-preserving, see  \cite{MY} and \cite{KS}.

\begin{definition}
Let $e:\Sigma\hookrightarrow \R^3$ be an embedding of a compact surface $\Sigma$ into $\R^3$. If a group $G$ acts on $\Sigma$ and $\R^3$ such that $h\circ e=e\circ h$ for each $h \in G$ and the $G$-action on $\R^3$ can be conjugated into $O(3)$, then we call such a group action on $\Sigma$ {\it extendable} over $\R^3$ with respect to $e$.
\end{definition}

For simplicity, we will say ``$G$ acts on the pair $(\mathbb{R}^3, \Sigma)$ " in the sense of this definition.

\begin{remark}
An orthogonal action on $\mathbb{R}^3$ fixes $0$ and extends orthogonally to $\mathbb{R}^4$ acting trivially on the new coordinate, so it fixes pointwise a line in $\mathbb{R}^4$ through $0$ and restricts to an orthogonal
action on $S^3  \subset \mathbb{R}^4$  with a fixed point. Vice versa, an orthogonal action on $S^3$ with a fixed point $P$ acts orthogonally on $\mathbb{R}^4$ , fixes $0$ and $P$ and pointwise the line in $\mathbb{R}^4$ through $0$
and $P$, and hence restricts to an orthogonal action on the $\mathbb{R}^3$ orthogonal to this line. So the actions on $(S^3, \Sigma)$ with at least one fixed point are the same as the actions on $(\mathbb{R}^3, \Sigma)$.
\end{remark}

Let $\Sigma_{g, b}$ denote the orientable compact surface with genus $g$ and $b$ boundary components, and for $g>0$ let $\Sigma^-_{g, b}$ denote the non-orientable compact surface with genus $g$ and $b$ boundary components. Note that $\Sigma_{g,0}$ is the same as $\Sigma_{g}$, and $\Sigma^-_{g, 0}$ is the connected sum of $g$ real projective planes. It is well known that each compact surface is either $\Sigma_{g, b}$ or $\Sigma^-_{g, b}$, the surfaces with $b=0$ give all closed surfaces, the surfaces with $b\neq 0$ give all compact bordered surfaces, and only $\Sigma^-_{g, 0}$ cannot be embedded into $\R^3$. Also note that $\a(\Sigma_{g, 0})=2g$.

\begin{definition}
For a fixed $g>1$, let $E_{g}$ be the maximum order of all extendable finite group actions on $\Sigma_g$ for all embeddings $\Sigma_g\hookrightarrow\R^3$; let $CE_{g}$ be the maximum order of all extendable cyclic group actions on $\Sigma_g$ for all embeddings $\Sigma_g\hookrightarrow\R^3$; if we require that the actions are orientation-preserving on $\R^3$ (i.e. the group action can be conjugated into $SO(3)$), then the maximum orders we get will be denoted by $E^o_{g}$ and $CE^o_{g}$ respectively.

For a fixed $\a>1$, let $EA_{\a}$ be the maximum order of all extendable finite group actions on $\Sigma$ for all embeddings $\Sigma\hookrightarrow\R^3$ among all bordered surfaces with $\a(\Sigma)=\a$; let $CEA_{\a}$ be the maximum order of all extendable cyclic group actions on $\Sigma$ for all embeddings $\Sigma\hookrightarrow\R^3$ among all bordered surfaces with $\a(\Sigma)=\a$; if we require that the actions are orientation-preserving on $\R^3$, then the maximum orders we get will be denoted by $EA^o_{\a}$ and $CEA^o_{\a}$ respectively.
\end{definition}

Note that in the above definition if $g\leq 1$ or $\a\leq 1$, then the maximum order will be infinite (consider sphere, torus, disk and annulus) and which is not interesting. Our first results are about the closed surfaces.

\begin{theorem}\label{CE}
For each $g>1$, (1) $CE^o_g$ is $g+1$; (2) ${CE}_g$ is $2g+2$ for even $g$ and $2g$ for odd $g$.
\end{theorem}

\begin{theorem}\label{E}
For each $g>1$, (1) $E^o_g$ is given in the following table;
\begin{center}
  \begin{tabular}{|c|c|}
  \hline $g$ & $E^o_g$ \\
  \hline $3$ & $12$ \\
  \hline $5, 7$ & $24$  \\
  \hline $11, 19, 21$ & $60$ \\
  \hline others & $2g+2$ \\\hline
  \end{tabular}
\end{center}
(2) $E_g=2E^o_g$ when $g\neq 21$, and $E_{21}=88$.
\end{theorem}

\begin{remark}
The embedded surfaces realizing $CE^o_g$ and $CE_g$ can be unknotted as well as knotted (where, viewing $S^3$ as the
one-point compactification of $\R^3$, a surface is unknotted if it bounds  handlebodies on both sides). This is also true for $E^o_g$ with $g\neq 21$. For $E^o_{21}$ the surfaces must be knotted. On the other hand, the embedded surfaces realizing $E_g$ must be unknotted. At this point, it would be worth to compare with those results in \cite{WWZZ2} and \cite{WWZ1}.
\end{remark}

Note that when $b\neq 0$, $\a(\Sigma_{g, b})=2g-1+b$ and $\a(\Sigma^-_{g, b})=g-1+b$.
Therefore there are many bordered surfaces having algebraic genus $\alpha$. For the bordered surfaces, we have the following results.

\begin{theorem}\label{CEA}
For each $\a>1$,

(1) $CEA^o_\a$ is $\a+1$, and the surfaces realizing $CEA^o_\a$ are $\Sigma_{0,\a+1}$ and $\Sigma_{\a/2,1}$ when $\a$ is even, and are $\Sigma_{0,\a+1}$ and $\Sigma_{(\a-1)/2,2}$ when $\a$ is odd;

(2) $CEA_\a$ is $2\a+2$ for even $\a$ and $2\a$ for odd $\a$, and the surface realizing ${CEA}_\a$ is $\Sigma_{0,\a+1}$ in both cases.
\end{theorem}

\begin{theorem}\label{EA}
For each $\a>1$, (1) $EA^o_\a$ and the surfaces realizing $EA^o_\a$ are listed below;
\begin{center}
  \begin{tabular}{|c|c|c|}
  \hline $\a$ & $EA^o_\a$ & $\Sigma$\\
  \hline $3$ & $12$ & $\Sigma_{0,4}$, $\Sigma^-_{1,3}$\\
  \hline $5$ & $24$ & $\Sigma_{0,6}$, $\Sigma_{1,4}$   \\
  \hline $7$ & $24$ & $\Sigma_{0,8}$, $\Sigma^-_{4,4}$  \\
  \hline $11$ & $60$ & $\Sigma_{0,12}$, $\Sigma^-_{6,6}$\\
  \hline $19$ & $60$ & $\Sigma_{0,20}$, $\Sigma_{4,12}$, $\Sigma^-_{10,10}$, $\Sigma^-_{14,6}$\\
  \hline $21$ & $60$ & $\Sigma_{5,12}$\\
  \hline $29$ & $60$ & $\Sigma_{0,30}$, $\Sigma_{5,20}$, $\Sigma_{9,12}$, $\Sigma_{14,2}$\\
  \hline others, $\a$ even & $2\a+2$ &$\Sigma_{0,\a+1}$, $\Sigma_{\a/2,1}$\\
  \hline others, $\a$ odd & $2\a+2$ &$\Sigma_{0,\a+1}$, $\Sigma_{(\a-1)/2,2}$\\\hline
  \end{tabular}
\end{center}
(2) ${EA}_\a$ and the surfaces realizing ${EA}_\a$ are listed below.
\begin{center}
  \begin{tabular}{|c|c|c|}
  \hline $\a$ & ${EA}_\a$& $\Sigma$\\
  \hline $3$ & $24$ &$\Sigma_{0,4}$\\
  \hline $5$ & $48$&$\Sigma_{0,6}$\\
  \hline $7$ & $48$&$\Sigma_{0,8}$\\
  \hline $11$ & $120$& $\Sigma_{0,12}$\\
  \hline $19$ & $120$& $\Sigma_{0,20}$\\
  \hline others & $4\a+4$ &$\Sigma_{0,\a+1}$\\\hline
  \end{tabular}
\end{center}
\end{theorem}

\begin{remark}
In the above theorems, all the group actions realizing the maximum orders must be faithful on the compact surfaces, except for the case of $CEA_\a$ with odd $\a$. In this case, the group action realizing $CEA_\a$ must be non-faithful on $\Sigma_{0,\a+1}$. If we require that the actions on surfaces are faithful, then $CEA_\a$ is $\a+1$ for odd $\a$ (Proposition \ref{pro:faithful action}). There can be various different surfaces realizing this maximum order.
\end{remark}

\begin{remark} Extending group actions on bordered surfaces seems to be addressed for the first time in the present note. A connection to knot theory is below: the group action on $(\R^3, \Sigma)$ is also a group action on $(\R^3, \Sigma, \partial\Sigma)$. If we view $\partial\Sigma$ as a link in $\R^3$ and $\Sigma$ as its Seifert surface, then Theorem \ref{EA} (1) provides many interesting examples of $(\R^3, \text{Seifert surface}, \text{link})$ with large symmetries. See Section \ref{Sec:bordered surface} for intuitive pictures.
\end{remark}

\begin{remark}
We note the paper \cite{CC} of the similar interest but quite different
content, which address when bordered surfaces $\Sigma$ can be embedded into
$\mathbb{R}^3$ so that those surface homeomorphisms permuting boundary components
of $\Sigma$ can extend over $\mathbb{R}^3$.
\end{remark}

Finally we give a brief description of the organization of the paper.

In Section \ref{Sec:preliminary} we list some facts about orbifolds and their coverings, as
well as automorphisms of small permutation groups, which will be used in
the proofs.

In Section \ref{Sec:closed surface} we prove Theorem \ref{CE} and Theorem \ref{E}. The upper bound
is obtained by applying the Riemann-Hurwitz formula and some preliminary,
but somewhat tricky, arguments on 2- and 3-dimensional orbifolds;
the equivariant Dehn¡¯s Lemma is not involved.

In Section \ref{Sec:bordered surface} we prove Theorem \ref{CEA} and Theorem \ref{EA}. The upper bounds in Theorem \ref{CE} and Theorem \ref{E} will be used to give the upper bounds in Theorem \ref{CEA} and Theorem \ref{EA}.
The most complicated part of this paper is to identify the topological types of those surfaces realizing the upper bounds and find simple representative embeddings for such bordered surfaces.

The examples in Sections 3 and 4 reaching the upper bound are visible
as expected, however several examples, including Example 3.2 (3) and those
in Figures 13 and 14, should be beyond the expectation of most people
(including some of the authors) before they were found.

In Section \ref{Sec:graph}, we give similar results for graphs in $\R^3$.


\section{Preliminaries}\label{Sec:preliminary}

In this section, we list some facts which will be used in the later proofs.

For orbifold theory, see \cite{Th}, \cite{Du} and \cite{BMP}. We give a brief description here. All of the {\it $n$-orbifolds} that we considered have the form $M/H$. Here $M$ is an $n$-manifold and $H$ is a finite group acting faithfully
on $M$. For each point $x\in M$, denote its stable subgroup by $St(x)$, its image in $M/H$ by $x'$. If the order $|St(x)|>1$, $x'$ is called a {\it singular point} with {\it index} $|St(x)|$, otherwise it is called a {\it regular point}. If we forget the singular set, then we get the topological {\it underlying space} $|M/H|$ of the orbifold $M/H$.

We make a convention that in the orbifold setting, for $X = \Sigma/G$, we
define $\partial X = \partial \Sigma/G$, the image of $\partial \Sigma$ under the group action (call it the real
boundary of $X$).

A simple picture we should keep in mind is the following: suppose that $G$ acts on $(\R^3, \Sigma)$ and let
$$\Gamma_G=\{x\in \R^3\mid\exists\, g\in G, g\neq \text{id}, s.t.\, gx=x\},$$
then $\Gamma_G/G$ is the singular set of the 3-orbifold $\R^3/G$, and $\Sigma/G$ is a 2-orbifold with singular set $\Sigma/G\cap \Gamma_G/G$.

The covering spaces and the fundamental group of an orbifold can also be defined. Moreover, there is an one-to-one correspondence between the orbifold covering spaces and the conjugacy classes of subgroups of the fundamental group, and regular covering spaces correspond to normal subgroups. In the following, automorphisms, covering spaces and fundamental groups always refer to the orbifold setting.

Note that an involution (peridical map of order 2) on $\R^3$ is conjugate to either a reflection (about a plane), or a $\pi$-rotation (about a line), or
an antipodal map (about a point).

\begin{theorem}[Riemann-Hurwitz formula]\label{Thm of RH formula}
Suppose that $\Sigma_g\rightarrow \Sigma_{g'}$ is a regular branched covering with transformation group $G$. Let $a_1, a_2, \cdots, a_k$ be the branched points in $\Sigma_{g'}$ having indices $q_1, q_2, \cdots, q_k$. Then
$$2-2g=|G|(2-2g'-\sum^k_{i=1}(1-\frac{1}{q_i})).$$
\end{theorem}

We say that a group $G$ acts on the pair $(\mathbb{R}^3, \Sigma$) orientation-reversingly if there exists some $g \in G$ such that $g$ acts on $\mathbb{R}^3$ orientation-reversingly.

\begin{lemma}\label{index 2}
Suppose that $G$ acts on $(\mathbb{R}^3, \Sigma_g)$ orientation-reversingly and let $G^o=\{g\in G\mid g\,\,\text{acts orientation-preservingly on}\,\, (\mathbb{R}^3,\Sigma_g)\}$, then $G^o$ is an index $2$ subgroup in $G$.
\end{lemma}

Since every fixed point free involution on compact manifold $X$ give a quotient manifold $X/\Z_2$, with  Euler characteristic $\chi(X/\Z_2)=\frac{1}{2}\chi(X)$, we have

\begin{lemma}\label{fixed point}
There is  no fixed point free involution on  a compact manifold $X$ with odd Euler characteristic $\chi(X)$.
\end{lemma}

\begin{lemma}\label{lem:cyclic}
Suppose that the cyclic group $\Z_{2n}\subset O(3)$ acts on $\mathbb{R}^3$ orientation-reversingly, then its induced $\Z_2$-action on $|\mathbb{R}^3/\Z_n|\cong \R^3$ conjugates to a reflection or an antipodal map, whose fixed point set intersects the singular line of $\mathbb{R}^3/\Z_n$ transversely at one point.
\end{lemma}

\begin{lemma}\label{LemOrbOri}
Suppose that $G$ acts on a compact surface $\Sigma$ such that each singular point in the orbifold $X=\Sigma/G$ is isolated and the underlying space $|X|$ is orientable, then $\Sigma$ is orientable.
\end{lemma}

\begin{proof}
A compatible local orientation system of $|X|$ can be lifted to a compatible local orientation system of $\Sigma$. Hence it gives an orientation of $\Sigma$.
\end{proof}

\begin{lemma}\label{lem:nonori}
Suppose that $G$ acts on an orientable compact surface $\Sigma$ such that the orbifold $X=\Sigma/G$ contains non-isolated singular points or the underlying space $|X|$ is non-orientable, then $G$ has an index $2$ subgroup.
\end{lemma}

\begin{proof}
There must be an element of $G$ reversing the orientation of $\Sigma$. Then the orientation-preserving elements of $G$ form an index $2$ subgroup.
\end{proof}

\begin{lemma}\label{LemOrbConps}
Suppose that $G$ acts on $\R^3$ and $X$ is a suborbifold of $\R^3/G$ such that $|X|$ is connected. Let $i: X\hookrightarrow\R^3/G$ be the inclusion map, then the preimage of $X$ in $\R^3$ has $[\pi_1(\R^3/G):i_*(\pi_1(X))]$ connected components.
\end{lemma}

\begin{proof}
Consider the covering space of $\R^3/G$ corresponding to the subgroup $H=i_*(\pi_1(X))$. It is $\R^3/H$. Then the inclusion map $i$ can be lifted to a map $\tilde{i}: X\rightarrow\R^3/H$. Since $i$ is an embedding, $\tilde{i}$ is also an embedding.
$$\xymatrix{
           & \mathbb{R}^3 \ar[1,0] &\\
           & \mathbb{R}^3/H\ar[1,0]&\\
     X\ar[0,1]^i\ar[ur]^{\tilde{i}}&\mathbb{R}^3/G&
}$$
Note that the lift $\tilde{i}$ has totally $[G:H]$ different choices, and different choices correspond to disjoint images in $\mathbb{R}^3/H$. For each given $\tilde{i}$, the induced homomorphism $\tilde{i}_*: \pi_1(X)\rightarrow\pi_1(\R^3/H)$ is surjective. Then by Lemma 2.10 in \cite{WWZZ2}, the preimage of $\tilde{i}(X)$ in $\R^3$ is connected. Hence the preimage of $X$ in $\R^3$ has $[G:H]=[\pi_1(\R^3/G):i_*(\pi_1(X))]$ connected components.
\end{proof}

Let $S_n$ denote the permutation group on the numbers $\{1, 2, 3,\cdots,n\}$, and let $A_n$ denote its alternating subgroup. Let $Aut(G)$ denote the automorphism group of a group $G$. For an element $x\in S_n$, let $ord(x)$ denote its order. The following algebraic facts will be used to identify the compact surfaces realizing the maximum orders.

\begin{lemma}\label{FACTS}
(1) If $\{x, y\}$ generates $A_4$ with $ord(x)=2$, $ord(y)=3$, then there exists $\sigma\in Aut(A_4)$ such that $\{\sigma(x),\sigma(y)\}=\{(12)(34), (123)\}$.

(2) If $\{x, y\}$ generates $S_4$ with $ord(x)=2$, $ord(y)=3$, then there exists $\sigma\in Aut(S_4)$ such that $\{\sigma(x),\sigma(y)\}=\{(12), (134)\}$.

(3) If $\{x, y\}$ generates $S_4$ with $ord(x)=2$, $ord(y)=4$, then there exists $\sigma\in Aut(S_4)$ such that $\{\sigma(x),\sigma(y)\}=\{(12), (1234)\}$.

(4) If $\{x, y\}$ generates $A_5$ with $ord(x)=2$, $ord(y)=3$, then there exists $\sigma\in Aut(A_5)$ such that $\{\sigma(x),\sigma(y)\}=\{(12)(34), (135)\}$.

(5) If $\{x, y\}$ generates $A_5$ with $ord(x)=2$, $ord(y)=5$, then there exists $\sigma\in Aut(A_5)$ such that one of the following two cases holds:\\ \centerline{$\{\sigma(x),\sigma(y)\}=\{(12)(34), (12345)\}$, $\{\sigma(x),\sigma(y)\}=\{(13)(24), (12345)\}$.}

(6) If $\{x, y\}$ generates $A_5$ with $ord(x)=3$, $ord(y)=3$, then there exists $\sigma\in Aut(A_5)$ such that $\{\sigma(x),\sigma(y)\}=\{(123), (145)\}$.

(7) If $\{x, y\}$ generates $A_5$ with $ord(x)=3$, $ord(y)=5$, then there exists $\sigma\in Aut(A_5)$ such that one of the following four cases holds:\\
\centerline{$\{\sigma(x),\sigma(y)\}=\{(123), (12345)\}$,
$\{\sigma(x),\sigma(y)\}=\{(132), (12345)\}$,}\\
\centerline{$\{\sigma(x),\sigma(y)\}=\{(124), (12345)\}$,
$\{\sigma(x),\sigma(y)\}=\{(142), (12345)\}$.}
\end{lemma}

\begin{proof}
Since $A_4$, $S_4$ and $A_5$ have small orders, all these facts can be checked by an elementary enumeration. We prove (5) as an example.

Note that $x$ must have the form $(ab)(cd)$ with $a, b, c, d \in \{1,2,3,4,5\}$ and $y$ must be a rotation of order $5$, so up to an automorphism we can assume that $y$ is $(12345)$ and $\{a,b,c,d\}=\{1,2,3,4\}$. Then there are three cases: $x$ is $(12)(34)$ or $(13)(24)$ or $(14)(23)$. However, $(14)(23)$ and $(12345)$ do not generate $A_5$, since they generate an order $10$ subgroup of $A_5$ which is a dihedral group. So we have the statement as in the lemma.
\end{proof}


\section{Closed surfaces in $\mathbb{R}^3$}\label{Sec:closed surface}

In this section, we first construct some extendable actions which will be used to realize $CE^o_g$, $CE_g$, $E^o_g$ and $E_g$. The symmetries are those people can feel in their daily life. Then we give the proofs of Theorem \ref{CE} and Theorem \ref{E}.

In the following examples, we will give a finite graph $\Gamma\subset\R^3$ and a group action on $(\R^3, \Gamma)$. The action will be orthogonal and keep $\Gamma$ invariant as a set. Then there exists a regular neighborhood of $\Gamma$, denoted by $N(\Gamma)$, such that: (1) $\partial N(\Gamma)$ is a smoothly embedded closed surface in $\R^3$; (2) the action keeps $\partial N(\Gamma)$ invariant; (3) the genus of $\partial N(\Gamma)$ is the same as the genus of $\Gamma$, defined as $g(\Gamma)=1-\chi(\Gamma)$, where $\chi(\Gamma)$ is the Euler characteristic of $\Gamma$. Then we get extendable actions on $\partial N(\Gamma)$.

\begin{example}[General case]\label{general}
(1) For each $g>1$, let $\Gamma^1_g\subset\mathbb{R}^3$ be a graph with $2$ vertices and $g+1$ edges as the left picture of Figure \ref{fig:general action} (for $g=2$). It has genus $g$. Then note that:

\begin{figure}[h]
\centerline{\scalebox{0.8}{\includegraphics{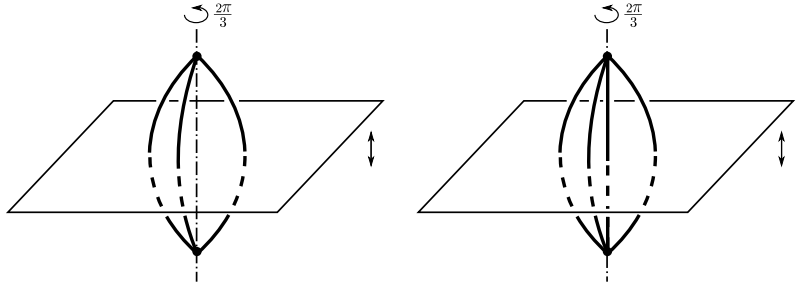}}}
\caption{Symmetric dipole graphs}\label{fig:general action}
\end{figure}
\noindent(i) there is a $2\pi/(g+1)$-rotation $\tau_{g+1}$ which generates a cyclic group of order $g+1$ acting on $(\R^3,\Gamma^1_g)$;\\
(ii) there is an orientation-preserving dihedral group $D_{g+1}$ action on $(\mathbb{R}^3,\Gamma^1_g)$, where $D_{g+1}$ is generated by $\tau_{g+1}$ and a $\pi$-rotation $\rho$ around the axis, which intersects an edge of $\Gamma^1_g$ and the rotation axis of $\tau_{g+1}$ in the plane;\\
(iii) when $g$ is even, the composition of the reflection $r$ about the plane and the rotation $\tau_{g+1}$ generates an order $2g+2$ cyclic group action on $(\mathbb{R}^3,\Gamma^1_g)$;\\
(iv) $\tau_{g+1}$, $\rho$ and $r$ generate a finite group of order $4g+4$ acting on $(\mathbb{R}^3,\Gamma^1_g)$.

(2) For odd $g$, let $\Gamma^2_g\subset\mathbb{R}^3$ be a graph with $2$ vertices and $g+1$ edges as the right picture of Figure \ref{fig:general action} (for $g=3$). It has genus $g$. Then the composition of the reflection $r$ about the plane and a $2\pi/g$-rotation generates an order $2g$ cyclic group action on $(\mathbb{R}^3,\Gamma^2_g)$.
\end{example}

\begin{example}[Special case]\label{special}
Let $T$, $C$, $O$, $D$, $I$ be the regular tetrahedron, cube, octahedron, dodecahedron, icosahedron. Their $1$-skeletons $T^{(1)}$, $C^{(1)}$, $O^{(1)}$, $D^{(1)}$, $I^{(1)}$ are graphs in $\R^3$ with genus $3$, $5$, $7$, $11$, $19$ respectively. See Figure \ref{fig:regular polyhedra}.
\begin{figure}[h]
\centerline{\scalebox{0.7}{\includegraphics{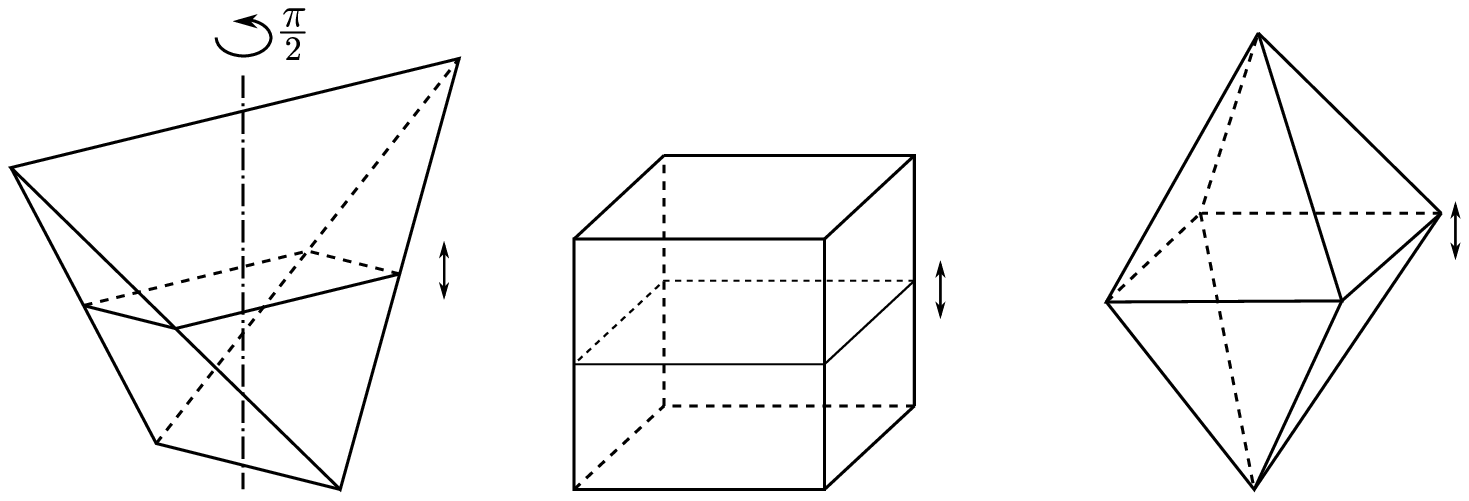}}}

\centerline{\scalebox{0.7}{\includegraphics{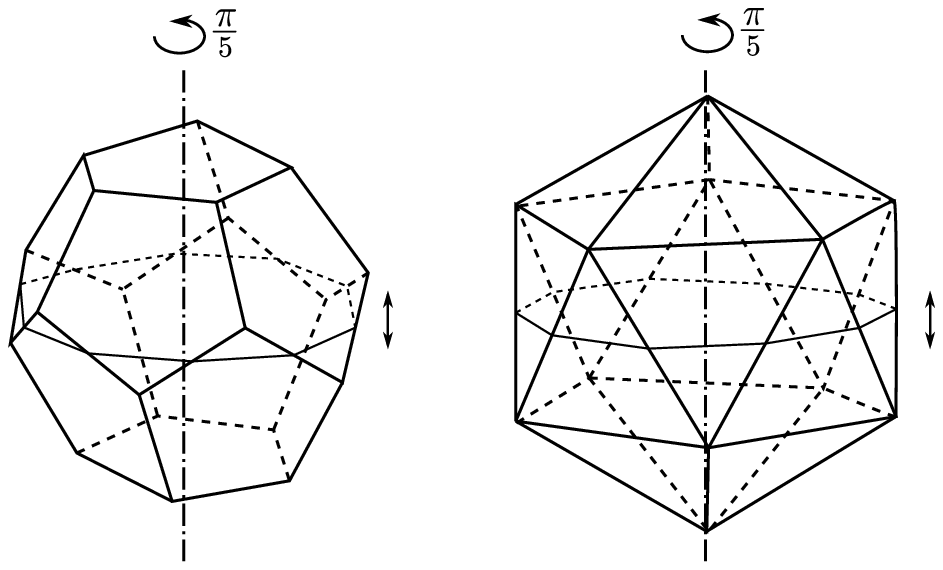}}}
\caption{Regular polyhedra}\label{fig:regular polyhedra}
\end{figure}

(1) The orientation-preserving isometry groups of the regular polyhedra keeps the graphs invariant. The group is $A_4$ of order $12$ for $T$, is $S_4$ of order $24$ for $C$ and $O$, and is $A_5$ of order $60$ for $D$ and $I$.

(2) The isometry groups of the regular polyhedra keeps the graphs invariant. These groups are obtained by adding orientation-reversing elements into the groups in (1), and have orders $24$, $48$, $48$, $120$, $120$ respectively. For $T$, the adding element can be the composition of a reflection and a $\pi/2$-rotation; for $C$ and $O$, the adding element can be a reflection; for $D$ and $I$, the adding element can be the composition of a reflection and a $\pi/5$-rotation. Actually, all these groups are generated by reflections.

(3) Consider an enlarged regular dodecahedron $D'$ with the same center as $D$. Let $v$ be a vertex of $D$ and let $v'$ be the vertex of $D'$ corresponding to $v$. Let $w$ be a vertex of $D'$ adjacent to $v'$. We can choose an arc connecting $v$ and $w$ such that its images under the orientation-preserving isometry group of $D$ only meet at vertices of $D$ and $D'$. Then the union of these image arcs is a connected graph with 40 vertices and 60 edges, therefore has genus $21$, and the orientation-preserving isometry group of $D$ keeps it invariant.
\end{example}

\begin{proposition}\label{pro:cyclic action}
Suppose that the cyclic group $\Z_n$ acts on $(\mathbb{R}^3,\Sigma_g)$ with $g>1$ orientation-preservingly, then one of the following holds:\\
(i) $n=g+1$, and $\Sigma_g/\Z_n$ is a sphere with $4$ singular points of index $n$;\\
(ii) $n=g$, and $\Sigma_g/\Z_n$ is a torus with $2$ singular points of index $n$;\\
(iii) $n=g-1$, and $\Sigma_g/\Z_n$ is a closed surface of genus $2$;\\
(iv) $n\leq g/2+1$.
\end{proposition}

\begin{proof}
Since the $\Z_n$-action can be conjugated into $SO(3)$, it is a rotation of order $n$ around a line. Then the 3-orbifold $\mathcal{O}=\mathbb{R}^3/\Z_n$ has underlying space $\mathbb{R}^3$ and singular set a line with index $n$. The 2-orbifold $X=\Sigma_g/\Z_n$, with underlying space a closed surface, must be separating in $\mathcal{O}$. Hence $X$ has even number of singular points. Suppose that $X$ has $2k'$ singular points and $|X|$ has genus $g'$. Then by the Riemann-Hurwitz formula we have
$$2-2g=n(2-2g'-2k'(1-\frac{1}{n}))=n(2-2g'-2k')+2k'.$$
Since $g>1$, $2-2g'-2k'<0$. Then $g'+k'-1\geq 1$, and
$$n=\frac{g+k'-1}{g'+k'-1}=\frac{g}{g'+k'-1}+\frac{k'-1}{g'+k'-1}.$$

If $g'+k'-1=1$, then $g'+k'=2$ and $n=g+k'-1$. When $k'=2$, $n=g+1$ and $g'=0$; when $k'=1$, $n=g$ and $g'=1$; when $k'=0$, $n=g-1$ and $g'=2$. If $g'+k'-1\geq 2$, then $n\leq g/2+1$.
\end{proof}

\begin{proposition}\label{pro:noncyclic}
When $g>1$ is odd, $\Z_{2g+2}$ and $\Z_{2g-2}$ can not act on $(\mathbb{R}^3,\Sigma_g)$ orientation-reversingly.
\end{proposition}

\begin{proof}
Suppose that $\Z_{2g+2}$ acts on $(\mathbb{R}^3,\Sigma_g)$ orientation-reversingly. Let $t$ be a generator of $\Z_{2g+2}$, then $t^2$ generates $\Z_{g+1}$. Let $X=\Sigma_g/\Z_{g+1}$, then by Proposition \ref{pro:cyclic action}, $X$ is a sphere with $4$ singular points, and it bounds a 3-orbifold $\mathcal{O}_1$ in $\R^3/\Z_{g+1}$ such that $|\mathcal{O}_1|$ is a 3-ball and the singular set of $\mathcal{O}_1$ consists of two arcs. By Lemma \ref{lem:cyclic}, $t$ induces an orientation-reversing $\Z_2$-action on $\mathcal{O}_1$, which can not have fixed point in each singular arc of $\mathcal{O}_1$. Then it has no singular fixed point. Since $|\mathcal{O}_1|$ has Euler characteristic $1$, by Lemma \ref{fixed point} there exists a regular fixed point $x\in\mathcal{O}_1$. Let $x'$ be a preimage of $x$ in $\R^3$, then the stable subgroup $St(x')\subset\Z_{2g+2}$ is isomorphic to $\Z_2$ and its generator $t^{g+1}$ is orientation-reversing. This is a contradiction since $g$ is odd. For $\Z_{2g-2}$, by Proposition \ref{pro:cyclic action}, the corresponding $X$ is a closed surface of genus $2$, and it bounds a 3-manifold $M$ in $\R^3/\Z_{g-1}$ such that $M$ has Euler characteristic $-1$. Then the induced $\Z_2$-action has a fixed point $x\in M$ by Lemma \ref{fixed point}, and we can get a contradiction as above since $g$ is odd.
\end{proof}

\begin{proof}[Proof of Theorem \ref{CE}]
(1) By Example \ref{general}(1)(i) $\Z_{g+1}$ acts on $(\mathbb{R}^3,\partial N(\Gamma^1_g))$ orientation-preservingly. Then by Proposition \ref{pro:cyclic action}, $CE^o_g=g+1$.

(2) By Example \ref{general}(1)(iii) and Example \ref{general}(ii), when $g$ is even, $\Z_{2g+2}$ acts on $(\mathbb{R}^3,\partial N(\Gamma^1_g))$ orientation-reversingly; when $g$ is odd, $\Z_{2g}$ acts on $(\mathbb{R}^3,\partial N(\Gamma^2_g))$ orientation-reversingly. Then by Theorem \ref{CE}(1), the extendable action reaching $CE_g$ must be orientation-reversing. Hence by Lemma \ref{index 2} and Proposition \ref{pro:noncyclic}, $CE_g$ is $2g+2$ for even $g$ and $2g$ for odd $g$.
\end{proof}

\begin{proof}[Proof of Theorem \ref{E}]
(1) By Example \ref{general}(1)(ii) and Example \ref{special}(1)(iii), we only need to show that if $G$ orientation-preservingly acts on $(\mathbb{R}^3,\Sigma_g)$, then $|G|$ is not bigger than the value given in Theorem \ref{E}. Since the $G$-action can be conjugated into $SO(3)$, by Theorem \ref{CE}(1) we can assume that $G$ is one of $D_n$, $A_4$, $S_4$ and $A_5$. Here $D_n$ is the dihedral group of order $2n$. The singular set of the corresponding orbifold $\mathcal{O}=\mathbb{R}^3/G$ is shown in Figure \ref{fig:singular set}. The underlying space of $\mathcal{O}$ is always $\R^3$, and the indices are the index numbers of the branch lines.

\begin{figure}[h]
\centerline{\scalebox{0.8}{\includegraphics{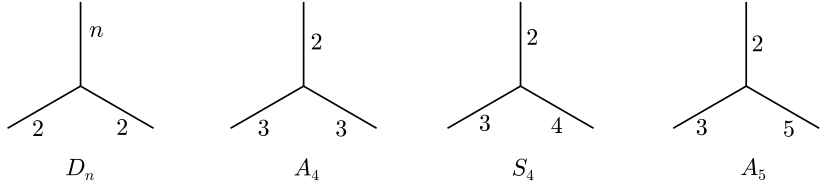}}}
\caption{Singular set of the orbifolds $\R^3/G$}\label{fig:singular set}
\end{figure}
The $2$-orbifold $X=\Sigma_g/G$ must be separating in $\mathcal{O}$, hence $X$ bounds a 3-orbifold $\mathcal{O}_1$ in $\mathcal{O}$ such that $|\mathcal{O}_1|$ is compact. Suppose that $X$ has $k$ singular points and $|X|$ has genus $g'$. Then by the Riemann-Hurwitz formula we have
$$2-2g=|G|(2-2g'-\sum^k_{i=1}(1-\frac{1}{q_i})).$$
We can assume that $|G|>2g-2$. Note that $q_i>1$ for each $i$, then
$$2g'+\frac{k}{2}\leq 2g'+\sum^k_{i=1}(1-\frac{1}{q_i})=2+\frac{2g-2}{|G|}<3.$$
Hence $4g'+k\leq 5$. If $g'=1$, then $k\leq 1$. Since $X$ can not intersect the singular set of $\mathcal{O}$ only once, we have $k=0$. Since $g>1$, this contradicts the Riemann-Hurwitz formula. Hence $g'=0$, $k\leq 5$ and $|\mathcal{O}_1|$ is a 3-ball.

If $\mathcal{O}_1$ contains the vertex of the singular set of $\mathcal{O}$, then each half line of the singular set of $\mathcal{O}$ must intersect $X$ odd times. Choose a intersection in each half line and assume that they have indices $q_1$, $q_2$ and $q_3$. Then for each case of $D_n$, $A_4$, $S_4$ and $A_5$ we have
$$\sum^3_{i=1}(1-\frac{1}{q_i})=2(1-\frac{1}{|G|}).$$
(One can check this directly, for example, for $A_4$, $\{q_1, q_2, q_3\}=\{2,3,3\}$ and $|A_4|=12$.)
The Riemann-Hurwitz formula now becomes
$$2g=|G|\sum^k_{i=4}(1-\frac{1}{q_i}).$$
Hence $k=5$. Then since $q_i>1$ for each $i$, we have $|G|\leq 2g$.

If $\mathcal{O}_1$ does not contain the vertex of the singular set of $\mathcal{O}$, then each half line of the singular set of $\mathcal{O}$ must intersect $X$ even times. Then we can assume that $k=2k'$ and $q_i=q_{i+k'}$ for $1\leq i\leq k'$. The Riemann-Hurwitz formula now becomes
$$2-2g=|G|(2-2\sum^{k'}_{i=1}(1-\frac{1}{q_i})).$$
Since $g>1$ and $q_i>1$ for each $i$, we have $k'=2$ and $k=4$. Hence $X$ is a sphere with $4$ singular points, and $\mathcal{O}_1$ is as Figure \ref{fig:3ball2arc}.

\begin{figure}[h]
\centerline{\scalebox{1}{\includegraphics{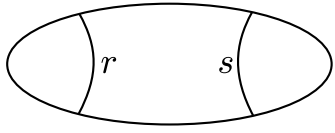}}}
\caption{3-ball with two singular arcs}\label{fig:3ball2arc}
\end{figure}
\noindent Here $r$ and $s$ are the indices of the singular arcs in $\mathcal{O}_1$. We can assume that $1<r\leq s$. The Riemann-Hurwitz formula can be rewritten as
$$g-1=|G|(1-\frac{1}{r}-\frac{1}{s}).$$
Since we assume that $|G|>2g-2$, we have
$$\frac{1}{r}+\frac{1}{s}>\frac{1}{2}.$$
If $G=D_n$, then $G$ contains $\Z_n$ as an index $2$ subgroup. By Theorem \ref{CE}(1), $n\leq g+1$. Hence $|G|\leq 2g+2$. The equality holds when $n=g+1$ and $(r,s)=(2,n)$. If $G=A_4$, then $|G|=12$ and $(r,s)$ is one of $(2,3)$ and $(3,3)$. Hence $g$ is $3$ or $5$. If $G=S_4$, then $|G|=24$ and $(r,s)$ is one of $(2,3)$, $(2,4)$, $(3,3)$ and $(3,4)$. Hence $g$ is one of $5$, $7$, $9$ and $11$. If $G=A_5$, then $|G|=60$ and $(r,s)$ is one of $(2,3)$, $(2,5)$, $(3,3)$ and $(3,5)$. Hence $g$ is one of $11$, $19$, $21$ and $29$. Then we get the results except the case of $g=9$.

If $G=S_4$ and $(r,s)=(3,3)$, then $\pi_1(X)$ is generated by elements of order $3$. Hence its image in $\pi_1(\mathcal{O})\cong G$ is contained in $A_4$. By Lemma \ref{LemOrbConps} the preimage of $X$ in $\R^3$ is not connected. However by our definition the preimage of $X$ is $\Sigma_g$ which is connected. The contradiction means that the above case of $g=9$ does not happen.

(2) By Example \ref{general}(1)(iv), Example \ref{special}(2), Theorem \ref{E}(1) and Lemma \ref{index 2}, we need only to consider the case of $g=21$. If there is a $G$-action on $(\mathbb{R}^3,\Sigma_{21})$ such that $|G|>88=4(21+1)$, then $|G^o|>2(21+1)$. By the proof of Theorem \ref{E}(1), $G^o=A_5$ and $\Sigma_{21}/G^o$ is a sphere with $4$ singular points of index $3$. Since the $G$-action can be conjugated into $O(3)$, it induces a reflection on $\R^3/G^o$, and the reflection plane contains the singular set of $\R^3/G^o$. The reflection plane cuts $\Sigma_{21}/G^o$ into two homeomorphic connected bordered surfaces. Hence the intersection of $\Sigma_{21}/G^o$ and the refection plane is one circle. Then the image of $\pi_1(\Sigma_{21}/G^o)$ in $\pi_1(\R^3/G^o)$ is isomorphic to $\Z_3$. By Lemma \ref{LemOrbConps} the preimage of $\Sigma_{21}/G^o$ in $\R^3$ is not connected, and we get a contradiction.
\end{proof}


\section{Bordered surfaces in $\mathbb{R}^3$}\label{Sec:bordered surface}

In this section, we first construct some extendable actions which will be used to realize $CEA^o_\a$, $CEA_\a$, $EA^o_\a$ and $EA_\a$. The examples mainly come from Example \ref{general} and Example \ref{special}. Then we give the proofs of Theorem \ref{CEA} and Theorem \ref{EA}.

\begin{example}\label{ex:bordered surface}
(1) For the graph in Example \ref{general}(1), we can replace each vertex with a disk and replace each edge with a band to get a bordered surface $\Sigma_{0,g+1}$ such that each of the group actions constructed in Example \ref{general}(i-iv) keeps $\Sigma_{0,g+1}$ invariant.

(2) For each of the graphs in Example \ref{special}(1)(2), we can replace each vertex with a disk and replace each edge with a band to get a bordered surface $\Sigma_{0,g+1}$, where $g$ is one of $3$, $5$, $7$, $11$, $19$, such that the corresponding group action in Example \ref{special}(1)(2) keeps $\Sigma_{0,g+1}$ invariant.

(3) For the graph in Example \ref{special}(3), we can replace each vertex with a disk and replace each edge with a band to get a bordered surface $\Sigma$ such that the group action in Example \ref{special}(3) keeps $\Sigma$ invariant.

Note that the replacement in (1-3) above does not change the fundamental groups. Since the genus of the graph equals the rank of its fundamental group, the algebraic genus of each surface in (1-3) equals the genus of the corresponding graph.
\end{example}

\begin{proposition}\label{le}
For each $\a>1$, we have $CEA^o_{\a}\le CE^o_{\a}$, $CEA_{\a}\le CE_{\a}$, $EA^o_{\a}\le E^o_{\a}$, $EA_{\a}\le E_{\a}$.
\end{proposition}

\begin{proof}
We only show that $EA_{\a}\le E_{\a}$. The proofs of the others are similar.

For any bordered surface $\Sigma\subset\R^3$ with $\a(\Sigma)=\a$ and any $G$-action on $(\R^3,\Sigma)$, the group $G$ also acts on $(\R^3,\partial N(\Sigma))$, where $N(\Sigma)$ is an equivariant regular neighborhood of $\Sigma$ such that $N(\Sigma)$ is a handlebody of genus $\a$ and $\partial N(\Sigma)$ is a smoothly embedded surface in $\R^3$. Then $\partial N(\Sigma)$ is homeomorphic to $\Sigma_\a$. Hence $|G|\le E_{\a}$, and $EA_{\a}\le E_{\a}$.
\end{proof}

\begin{remark}
By comparing Theorem \ref{CE}, \ref{E}, \ref{CEA} and \ref{EA}, the inequalities in Proposition \ref{le} are all sharp.
\end{remark}

\begin{proof}[Proof of Theorem \ref{CEA}]
(1) By Example \ref{ex:bordered surface}(1), there exists a $\Z_{\a+1}$-action on $(\R^3,\Sigma_{0,\a+1})$, and $\a(\Sigma_{0,\a+1})=\a$. By Proposition \ref{le} and Theorem \ref{CE}(1), we have $CEA^o_\a\leq CE^o_\a=\a+1$. So $CEA^o_\a=\a+1$. Following we need to determine all $\Sigma$ realizing $CEA^o_\a$.

Suppose that $\Z_{\a+1}$ acts on $(\R^3,\Sigma)$ with $\a(\Sigma)=\a$. Let $X=\Sigma/\Z_{\a+1}$, and let $N(\Sigma)$ be the regular neighborhood of $\Sigma$ as in the proof of Proposition \ref{le}, then $N(\Sigma)/\Z_{\a+1}$ is a regular neighborhood of $X$, denoted by $N(X)$. By Proposition \ref{pro:cyclic action}, $N(X)$ must be a 3-ball with $2$ singular arcs of index $\a+1$. Since $|N(X)|=N(|X|)$, $|X|$ must be a disk. Then since $\a+1>2$, the boundary points of $|X|$ are regular in $X$, and $X$ is a disk with $2$ singular points of index $\a+1$. By Lemma \ref{LemOrbOri}, $\Sigma$ is orientable, and we can assume that $\Sigma=\Sigma_{g,b}$.

Given the singular line in $\R^3/\Z_{\a+1}$ an orientation, there are two ways that $X$ intersects the singular line, as Figure \ref{fig:intersectionXandSline}. Let $\Z_{\a+1}=\langle t\mid t^{\a+1}\rangle$, and let $i: X\hookrightarrow\R^3/\Z_{\a+1}$ be the inclusion map. Consider the induced homomorphism $i_*: \pi_1(\partial X)\rightarrow\pi_1(\R^3/\Z_{\a+1})=\Z_{\a+1}.$
\begin{figure}[h]
\centerline{\scalebox{1}{\includegraphics{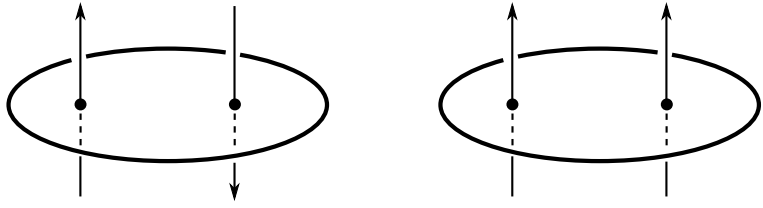}}}
\caption{Intersection of $X$ and the singular line}\label{fig:intersectionXandSline}
\end{figure}

In the left picture, the singular line goes through the two singular points in opposite directions. Then $i_*(\pi_1(\partial X))$ is trivial in $\Z_{\a+1}$. So
$$[\pi_1(\R^3/\Z_{\a+1}):i_*(\pi_1(\partial X))]=\a+1.$$
By Lemma \ref{LemOrbConps}, the preimage of $\partial X$ in $\R^3$ has $\a+1$ connected components. Hence $b=\a+1$. Since $\a(\Sigma_{g,b})=2g-1+b$, we have $g=0$ and $\Sigma=\Sigma_{0,\a+1}$.

In the right picture, the singular line goes through the two singular points in the same direction. Then $i_*(\pi_1(\partial X))$ in $\Z_{\a+1}$ is generated by $t^2$. So
$$[\pi_1(\R^3/\Z_{\a+1}):i_*(\pi_1(\partial X))]=((-1)^{\a+1}+3)/2.$$
By Lemma \ref{LemOrbConps}, the preimage of $\partial X$ in $\R^3$ is connected when $\a$ is even and has $2$ connected components when $\a$ is odd. Since $\a(\Sigma_{g,b})=2g-1+b$, we have $\Sigma=\Sigma_{\a/2,1}$ for even $\a$ and $\Sigma=\Sigma_{(\a-1)/2,2}$ for odd $\a$. An intuitive view of the surfaces is shown as Figure \ref{fig:ECaction on BSurface} (for $\a=2$).
\begin{figure}[h]
\centerline{\scalebox{1}{\includegraphics{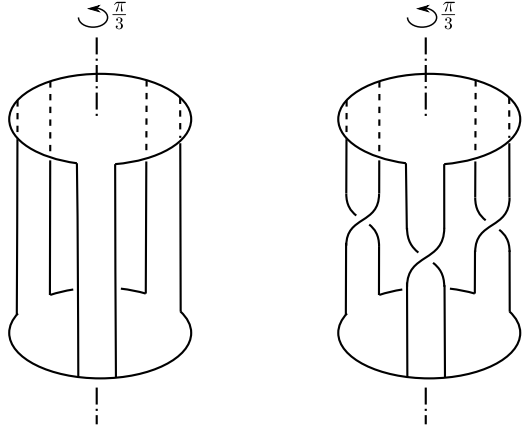}}}
\caption{Extendable cyclic action on bordered surfaces}\label{fig:ECaction on BSurface}
\end{figure}

(2) When $\a$ is even, by Example \ref{ex:bordered surface}(1) there exists a $\Z_{2\a+2}$-action on $(\R^3, \Sigma_{0,\a+1})$, and $\a(\Sigma_{0,\a+1})=\a$. By Proposition \ref{le} and Theorem \ref{CE}(2), we have
$CEA_\a\leq CE_\a=2\a+2$. So $CEA_\a$ is $2\a+2$ for even $\a$. Following we need to determine all $\Sigma$ realizing $CEA_\a$ with even $\a$.

Suppose that $\Z_{2\a+2}$ acts on $(\mathbb{R}^3,\Sigma)$ with $\a(\Sigma)=\a$. Let $t$ be a generator of $\Z_{2\a+2}$, then $t^2$ generates $\Z_{\a+1}$. Let $X=\Sigma/\Z_{\a+1}$. Then by the proof of Theorem \ref{CEA}(1), $X$ is a disk with two singular points. By Lemma \ref{lem:cyclic}, the $\Z_2$-action on $X$ induced by $t$ can not fix both the singular points of $X$. Hence it has no singular fixed points in $X$, and there exists regular fixed points in $X$. Then the $\Z_2$-action must be a reflection on $\R^3/\Z_{\a+1}$ and $X$. Since the $\Z_2$-action changes the orientation of the singular line in $\R^3/\Z_{\a+1}$, the singular line goes through the two singular points of $X$ in opposite directions, see Figure \ref{fig:reflection on X}. So by the proof of Theorem \ref{CEA}(1), $\Sigma=\Sigma_{0,\a+1}$.
\begin{figure}[h]
\centerline{\scalebox{1}{\includegraphics{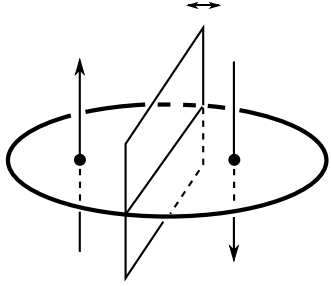}}}
\caption{Reflection on $X=\Sigma/\Z_{\a+1}$}\label{fig:reflection on X}
\end{figure}

When $\a$ is odd, by Proposition \ref{le} and Theorem \ref{CE}(2),
$CEA_\a\leq 2\a$. On the other hand, there exists a $\Z_{2\a}$-action on $(\R^3,\Sigma_{0,\a+1})$ indicated by the right picture of Figure \ref{fig:non-faithful action} (for $\a=3$). The surface lies on a plane, and the action is generated by a $2\pi/\a$-rotation and the reflection about the plane. Then $CEA_\a$ is $2\a$ for odd $\a$. Following we need to determine all $\Sigma$ realizing $CEA_\a$ with odd $\a$.
\begin{figure}[h]
\centerline{\scalebox{1}{\includegraphics{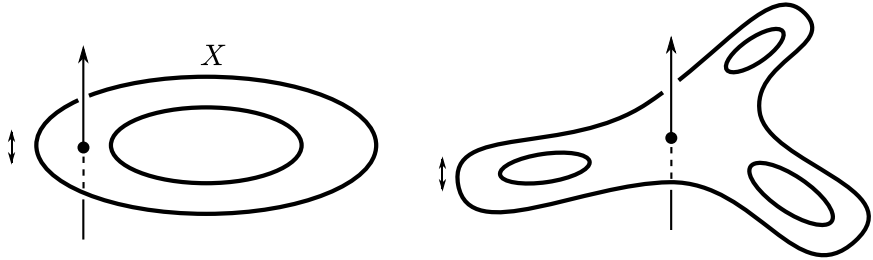}}}
\caption{Non-faithful extendable cyclic action on $\Sigma_{0,\alpha+1}$}\label{fig:non-faithful action}
\end{figure}

Suppose that $\Z_{2\a}$ acts on $(\mathbb{R}^3,\Sigma)$ with $\a(\Sigma)=\a$. Let $t$ be a generator of $\Z_{2\a}$, then $t^2$ generates $\Z_{\a}$. Let $X=\Sigma/\Z_{\a}$, and let $N(X)$ be the regular neighborhood of $X$ as in the proof of Theorem \ref{CEA}(1). By Proposition \ref{pro:cyclic action}, $\partial N(X)$ is a torus with two singular points of index $\a$, since $\a\geq 3$. Then $N(X)\subset\R^3/\Z_\a$ must be a solid torus with one singular arc, and $|X|$ must be an annulus or a M\"obius band. Since $\a\geq 3$, the boundary points of $|X|$ are regular in $X$, and $X$ has exactly $1$ singular point $w$ of index $\a$. Then the involution induced by $t$ fixes this singular point $w$; since $\chi(X- w)=-1$,  by Lemma \ref{fixed point} the involution must have other  regular fixed points in $X$. By Lemma \ref{lem:cyclic}, the $\Z_2$-action on $\R^3/\Z_{\a}$ is a reflection, and the reflection plane contains the singular point of $X$. Then $X$ can not intersect the reflection plane transversely. Hence $X$ lies on the reflection plane, and $|X|$ is an annulus, as the left picture of Figure \ref{fig:non-faithful action}. By Lemma \ref{LemOrbOri}, $\Sigma$ is orientable, and we can assume that $\Sigma=\Sigma_{g,b}$.

In this case $\partial X$ has two components. The fundamental group of one component is mapped to a generator of $\Z_\a$, and the fundamental group of the other component is mapped to the identity element of $\Z_\a$. By Lemma \ref{LemOrbConps}, the preimages of the two components in $\R^3$ have $1$ and $\a$ components respectively. Hence $b=\a+1$ and $\Sigma=\Sigma_{0,\a+1}$.

Note that in this case the bordered surface lies on a plane, and the action on the surface must be non-faithful.
\end{proof}

\begin{proposition}\label{pro:faithful action}
If we require that the actions on surfaces are faithful in the definition of $CEA_\a$, then $CEA_\a$ is $\a+1$ for odd $\a$.
\end{proposition}

\begin{proof}
By Theorem \ref{CEA}(1), we only need to consider orientation-reversing group actions. Suppose that $\Z_{2n}$ acts on $(\R^3,\Sigma)$ with $\a(\Sigma)=\a$ orientation-reversingly. Let $N(\Sigma)$ be the regular neighborhood of $\Sigma$ as in the proof of Proposition \ref{le}, then $\Z_{2n}$ acts on $(\R^3,\partial N(\Sigma))$, and $\Z_n\subset\Z_{2n}$ acts on $(\R^3,\partial N(\Sigma))$ orientation-preservingly. By Proposition \ref{pro:noncyclic}, $n$ can not be $\a+1$ and $\a-1$. By the proof of Theorem \ref{CEA}(2), $n$ can not be $\a$, since we require that the actions on surfaces are faithful. Then by proposition \ref{pro:cyclic action}, $n\leq \a/2+1$. Since $\a$ is odd, $n\leq(\a+1)/2$ and $2n\leq\a+1$.
\end{proof}

\begin{proof}[Proof of Theorem \ref{EA}]
(1) By Example \ref{ex:bordered surface}, Proposition \ref{le} and Theorem \ref{E}, we have $EA^o_\a=E^o_\a$. Hence we have the orders in the table. Following we need to determine all $\Sigma$ realizing $EA^o_\a$.

Suppose that $G$ acts on $(\R^3,\Sigma)$ with $\a(\Sigma)=\a$, and $|G|=EA^o_\a$.  Let $X=\Sigma/G$, and let $N(X)$ be the regular neighborhood of $X$ as in the proof of Theorem \ref{CEA}(1). By the proof of Theorem \ref{E}, $N(X)$ must be a 3-ball with $2$ singular arcs of indices $r$ and $s$, where $1<r\leq s$. Since $|N(X)|=N(|X|)$, $|X|$ must be a disk. The possible cases of $G$ and $(r,s)$ are listed below:\\
(i) $\a>1$: $G=D_{\a+1}$ and $(r,s)=(2,\a+1)$;\\
(ii) $\a=3$: $G=A_4$ and $(r,s)=(2,3)$;\\
(iii) $\a=5$: $G=S_4$ and $(r,s)=(2,3)$;\\
(iv) $\a=7$: $G=S_4$ and $(r,s)=(2,4)$;\\
(v) $\a=11$: $G=A_5$ and $(r,s)=(2,3)$;\\
(vi) $\a=19$: $G=A_5$ and $(r,s)=(2,5)$;\\
(vii) $\a=21$: $G=A_5$ and $(r,s)=(3,3)$;\\
(viii) $\a=29$: $G=A_5$ and $(r,s)=(3,5)$.

Then there are two possibilities of $X$ as in Figure \ref{fig:two possibilities}.

Case (a): the boundary points of $|X|$ are regular in $X$, and $X$ is a disk with $2$ singular points of indices $r$ and $s$. By Lemma \ref{LemOrbOri}, $\Sigma$ is orientable, and we can assume that $\Sigma=\Sigma_{g,b}$.

Case (b):    $\partial |X|=\gamma_1\cup\gamma_2$, $\gamma_1$ is
the real boundary,  and $\gamma_2$ is the reflection boundary.
$X$ is a disk with  a singular arc $\gamma_2$ of index $2$  on $\partial |X|$ and a singular point of index $s$ in the interior of $X$.
Note  $\partial X=\gamma_1$, the bold arc  in the right of Figure \ref{fig:two possibilities}.
In this case,  $\Sigma$ can be orientable as well as non-orientable, we assume that $\Sigma=\Sigma_{g,b}$ or $\Sigma=\Sigma^-_{g,b}$.

\begin{figure}[h]
\centerline{\scalebox{1}{\includegraphics{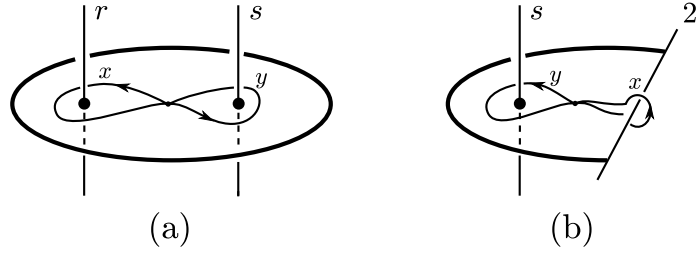}}}
\caption{Two possibilities of $X$}\label{fig:two possibilities}
\end{figure}

Let $i: X\hookrightarrow\R^3/G$ be the inclusion map. Since the preimage of $X$ in $\R^3$ is $\Sigma$, which is connected, by Lemma \ref{LemOrbConps}, $i_*$ is surjective. Note that $\pi_1(X)$ is isomorphic to the free product of $\Z_r$ and $\Z_s$, which correspond to the two singular points in case (a) and correspond to the reflection boundary arc and the singular point in case (b). Let $u$ be a generator of $\Z_r$, let $v$ be a generator of $\Z_s$, and let $x=i_*(u)$, $y=i_*(v)$. Then $x$ has order $r$, $y$ has order $s$, and $\{x,y\}$ generates the group $G$.

In case (a), $i_*(\pi_1(\partial X))$ is generated by $xy$.
In case (b), let $N(\partial X)=N(\gamma_1)$ be the regular neighborhood of $\partial X$ in $\R^3/G$.
Then $N(\partial X)$ is a 3-ball ($|N(\partial X)|=B^3$) wiht two index $2$ lines in it,
 and it is easy to see
that
$i_*(\pi_1(\partial X))=i_*(\pi_1(N(\partial X)))$ is generated by $\{x,y^{-1}xy\}$. Hence by Lemma \ref{LemOrbConps}, $b=[G:\langle xy\rangle]$ in case (a), and $b=[G:\langle x,y^{-1}xy\rangle]$ in case (b). Note that $\Sigma$ is always orientable in case (a), in each case of (i-viii) we will determine whether $\Sigma$ is orientable in case (b), then by $\a(\Sigma_{g,b})=2g-1+b$ and $\a(\Sigma^-_{g,b})=g-1+b$, we can get the genus $g$.

Below we identify $\Sigma$ case by case. Except for the cases of (vi) and (viii), all the surfaces are those in Example \ref{ex:bordered surface} or can be constructed from those in Example \ref{ex:bordered surface} by replacing each band by a half-twisted band.

(i) Note that $D_{\a+1}=\langle a,b \mid a^2, b^{\a+1}, (ab)^2\rangle$. We can assume that $x=a$ and $y=b$.

Case (a): $b=[G:\langle ab\rangle]=\a+1$. Hence $\Sigma=\Sigma_{0,\a+1}$.

Case (b): if $\a$ is even, then $b=[G:\langle a,b^{-1}ab\rangle]=1$; if $\a$ is odd, then $b=[G:\langle a,b^{-1}ab\rangle]=2$. Since $D_{\a+1}$ contains $\Z_{\a+1}$ as an index $2$ subgroup, $\R^3/\Z_{\a+1}$ is a $2$-sheet regular orbifold covering space of $\R^3/D_{\a+1}$, and in $\R^3/\Z_{\a+1}$ there is a $2$-sheet regular orbifold covering space of $X$, denoted by $X'$. Then $X'$ is a disk with $2$ singular points of index $\a+1$. By Lemma \ref{LemOrbOri}, $\Sigma$ is orientable. Hence $\Sigma=\Sigma_{\a/2,1}$ for even $\a$ and $\Sigma=\Sigma_{(\a-1)/2,2}$ for odd $\a$.

An intuitive view of the surfaces is as Figure \ref{fig:ECaction on BSurface}.

(ii) By Lemma \ref{FACTS}(1), we can assume that $\{x,y\}=\{(12)(34),(123)\}$.

Case (a): $b=[A_4:\langle(134)\rangle]=4$. Hence $\Sigma=\Sigma_{0,4}$.

Case (b): $b=[A_4:\langle(12)(34),(14)(23)\rangle]=3$. If $\Sigma$ is orientable, then $\a=2g-1+3$ is even, which is a contradiction. Hence $\Sigma$ is non-orientable, and $\Sigma=\Sigma^-_{1,3}$.

An intuitive view of the surfaces can be seen in Figure \ref{fig:g3nonori}(see also Figure
3 of \cite{CC}).

\begin{figure}[h]
\centerline{\scalebox{0.7}{\includegraphics{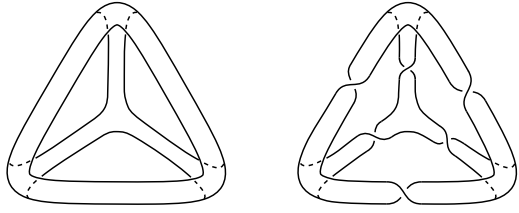}}}
\caption{$\Sigma_{0,4}$ and $\Sigma^-_{1,3}$}\label{fig:g3nonori}
\end{figure}

(iii) By Lemma \ref{FACTS}(2), we can assume that $\{x,y\}=\{(12),(134)\}$.

Case (a): $b=[S_4:\langle(1234)\rangle]=6$. Hence $\Sigma=\Sigma_{0,6}$.

Case (b): $b=[S_4:\langle(12),(23)\rangle]=4$. Consider the $2$-sheet covering space of $X$ in $\mathbb{R}^3/A_4$, denoted by $X'$, as Figure \ref{fig:double cover}. Then $X'$ is a disk with $2$ singular points of index $3$. By Lemma \ref{LemOrbOri}, $\Sigma$ is orientable. Hence $\Sigma=\Sigma_{1,4}$.
\begin{figure}[h]
\centerline{\scalebox{1.2}{\includegraphics{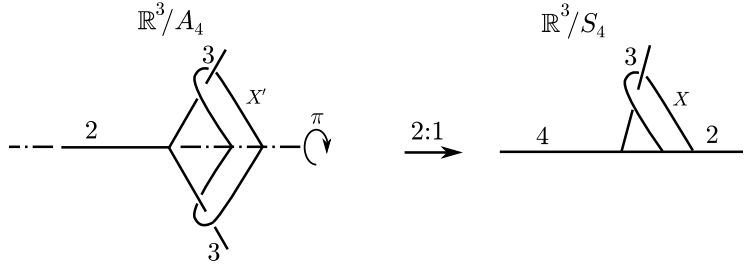}}}
\caption{2-sheet covering with $|X'|$ a disk}\label{fig:double cover}
\end{figure}

(iv) By Lemma \ref{FACTS}(3), we can assume that $\{x,y\}=\{(12),(1234)\}$.

Case (a): $b=[S_4:\langle(134)\rangle]=8$. Hence $\Sigma=\Sigma_{0,8}$.

Case (b): $b=[S_4:\langle(12),(23)\rangle]=4$. Consider the $2$-sheet covering space of $X$ in $\R^3/A_4$, denoted by $X'$, as Figure \ref{fig:double}. If we view $|X'|$ as a bordered surface in $|\R^3/A_4|$, then by Lemma \ref{LemOrbConps}, $|X'|$ has one boundary component. Hence $X'$ is a M\"obius band with one singular point of index $2$ in $\R^3/A_4$. Since $A_4$ has no index $2$ subgroups, by Lemma \ref{lem:nonori}, $\Sigma$ is non-orientable. Hence $\Sigma=\Sigma^-_{4,4}$.
\begin{figure}[h]
\centerline{\scalebox{1.2}{\includegraphics{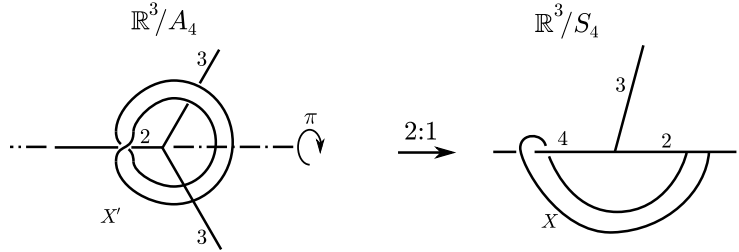}}}
\caption{2-sheet covering with $|X'|$ a M\"obius band}\label{fig:double}
\end{figure}

(v) By Lemma \ref{FACTS}(4), we can assume that $\{x,y\}=\{(12)(34),(135)\}$.

Case (a): $b=[A_5:\langle(12345)\rangle]=12$. Hence $\Sigma=\Sigma_{0,12}$.

Case (b): $b=[A_5:\langle(12)(34),(23)(45)\rangle]=6$. Since $A_5$ has no index $2$ subgroups, by Lemma \ref{lem:nonori}, $\Sigma$ is non-oriented. Hence $\Sigma=\Sigma^-_{6,6}$.

(vi) By Lemma \ref{FACTS}(5), we can assume that one of the following two cases holds: $\{x,y\}=\{(12)(34),(12345)\}$, $\{x,y\}=\{(13)(24),(12345)\}$.

Case (a): in the first case $b=[A_5:\langle(135)\rangle]=20$, and $\Sigma=\Sigma_{0,20}$; in the second case, $b=[A_5:\langle(14325)\rangle]=12$, and $\Sigma_{4,12}$.

Case (b): in the first case $b=[A_5:\langle(12)(34),(23)(45)\rangle]=6$; in the second case $b=[A_5:\langle(13)(24),(24)(35)\rangle]=10$. Since $A_5$ has no index $2$ subgroups, by Lemma \ref{lem:nonori}, $\Sigma$ is non-orientable. Hence $\Sigma=\Sigma^-_{14,6}$ or $\Sigma^-_{10,10}$.

An intuitive view of the surfaces is as following: $\Sigma=\Sigma_{0,20}$ is obtained by replacing vertices and edges of the icosahedron by disks and bands, as in Example \ref{ex:bordered surface}. Then $\Sigma=\Sigma^-_{14,6}$ can be obtained by replacing each band of $\Sigma_{0,20}$ by a half-twisted band. The other two surfaces are somehow non-trivial. An embedding of $\Sigma_{4,12}$ is as the left picture of Figure \ref{fig:gen19surface}. Each boundary of it is a $(5,2)$-torus knot, and the red circle is one of the $12$ boundaries. Then $\Sigma^-_{10,10}$ can be obtained by replacing the bands of $\Sigma_{4,12}$ by half-twisted bands as the right picture of Figure \ref{fig:gen19surface}.
\begin{figure}[h]
\centerline{\scalebox{1}{\includegraphics{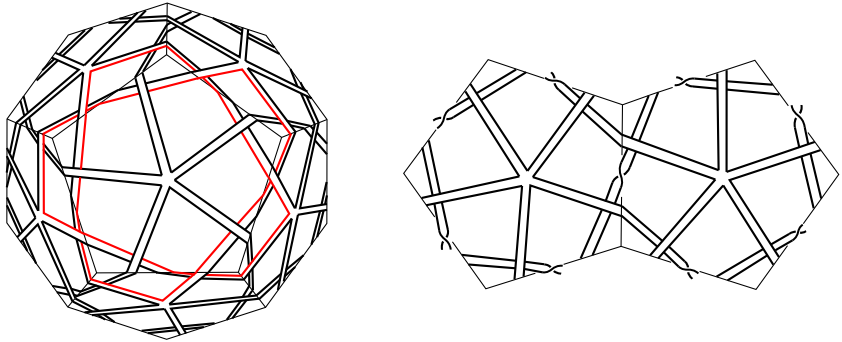}}}
\caption{$\Sigma_{4,12}$ and $\Sigma^-_{10,10}$}\label{fig:gen19surface}
\end{figure}

(vii) By Lemma \ref{FACTS}(6), we can assume that $\{x,y\}=\{(123),(145)\}$. There is only one case. $b=[A_5:\langle(12345)\rangle]=12$. Hence $\Sigma=\Sigma_{5,12}$.

(viii) By Lemma \ref{FACTS}(7), we can assume that one of the following four cases holds:\\
\centerline{$\{x,y\}=\{(123),(12345)\}$, $\{x,y\}=\{(132),(12345)\}$,}\\ \centerline{$\{x,y\}=\{(124),(12345)\}$, $\{x,y\}=\{(142),(12345)\}$.}

There is only one case of $X$, $xy$ is one of $(13245)$, $(145)$, $(13425)$, $(15)(34)$. Then $b$ is one of $12$, $20$, $12$, $30$. Since $\Sigma$ is always orientable, $\Sigma$ is one of $\Sigma_{9,12}$, $\Sigma_{5,20}$, $\Sigma_{9,12}$, $\Sigma_{0,30}$.

An intuitive view of the surfaces is as Figure \ref{fig:g29sur}. $\Sigma_{0,30}$ is as (a); the two $\Sigma_{9,12}$ are as (b) and (c); $\Sigma_{5,20}$ is as (d). (c) differs from (a) by the twists on each face, and (d) differs from (b) by the twists on each face.
\begin{figure}[h]
\centerline{\scalebox{1}{\includegraphics{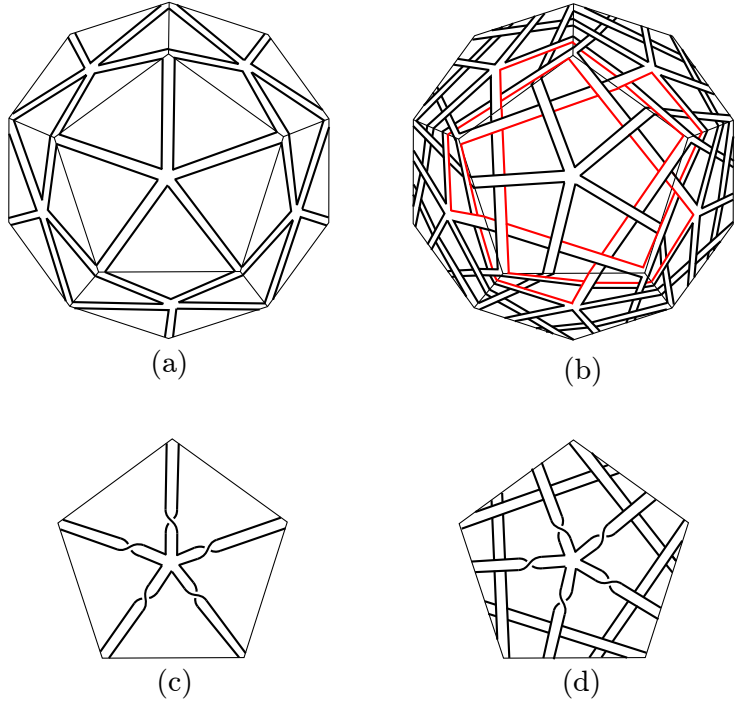}}}
\caption{Surfaces realizing $EA^o_{29}$}\label{fig:g29sur}
\end{figure}

\begin{figure}[h]
\centerline{\scalebox{0.7}{\includegraphics{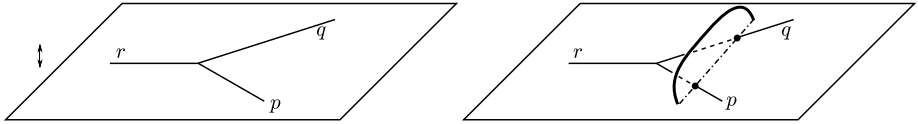}}}
\caption{Reflection plane of $\mathcal{O}$}\label{fig:reflection plane}
\end{figure}
(2) By Example \ref{ex:bordered surface}, Proposition \ref{le} and Theorem \ref{E}, $EA_\a=E_\a$. Hence we have the orders in the table. Following we need to determine all $\Sigma$ realizing $EA_\a$.

Suppose that $G$ acts on $(\R^3, \Sigma)$ realizing the maximum order, then the $G$-action is orientation-reversing. Hence the orientation-preserving elements of $G$ form an index $2$ subgroup $G^o$. Let $\mathcal{O}=\R^3/G^o$ and $X=\Sigma/G^o$. Then the $G$-action induces an orientation-reversing $\Z_2$-action on $\mathcal{O}$. By the proof of Theorem \ref{EA}(1), $|\mathcal{O}|\cong\mathbb{R}^3$, the singular set of $\mathcal{O}$ consists of $3$ half lines, and the induced $\Z_2$-action is a reflection on $|\mathcal{O}|$. The reflection plane $\Pi$ contains the singular set of $\mathcal{O}$, as the left picture of Figure \ref{fig:reflection plane}.

By the proof of Theorem \ref{EA}(1), there are two possibilities: case (a) and case (b). In case (b), the reflection boundary arc of $X$ lies on $\Pi$, hence the whole $X$ lies on $\Pi$. This is a contradiction since $X$ also contains a singular point in the interior of $|X|$. In case (a), $X$ has $2$ singular points which lie on $\Pi$. Since $X$ can not lie on $\Pi$, it must intersect $\Pi$ transversely along an arc passing the two singular points, as the right picture of Figure \ref{fig:reflection plane}. If the two singular points lie on the same singular line, then the preimage of $X$ can not be connected (comparing with the last paragraph of the proof of Theorem \ref{E}). Hence (vii) in the proof of Theorem \ref{EA}(1) does not happen. Then the singular points of $X$ lie on different singular lines, and the order of $xy$ in the proof of Theorem \ref{EA}(1) will be the same as the index of the singular line of $\mathcal{O}$ which does not intersect $X$. Then $\Sigma$ must be a punctured sphere.
\end{proof}


\section{A remark on graphs in $\mathbb{R}^3$}\label{Sec:graph}

Note that for a finite graph, its genus defined at the beginning of Section \ref{Sec:closed surface} coincides with its algebraic genus, defined as the rank of its fundamental group. We can define extendable group actions on graphs like the case of compact bordered surfaces, and define the maximum orders $EG_\a$, $CEG_\a$, $EG^o_\a$, $CEG^o_\a$ similarly. Then we have:

\begin{theorem}
For each $\a>1$, we have $CEG^o_\a=CEA^o_\a$, $EG^o_\a=EA^o_\a$, $CEG_\a=CEA_\a$, $EG_\a=EA_\a$.
\end{theorem}

\begin{proof}
All the examples we constructed come from the equivariant graphs in $\R^3$. So all these examples also apply for graphs with the same genera.

On the other hand, if a group $G$ acts on $(\mathbb{R}^3, \Gamma)$ for a graph $\Gamma$ of genus $g$, then $G$ also acts on $\partial N(\Gamma)$, as described in the beginning of Section \ref{Sec:closed surface}. This completes the proof.
\end{proof}

\begin{remark}
The maximum order of finite group actions on minimal graphs (i.e, the graphs without free edges) of genus $\a$ is $2^\a \a!$ if $\a>2$ and is 12 if $\a=2$ \cite{WZi}.
\end{remark}

\noindent Chao Wang, Jonsvannsveien 87B, H0201, Trondheim 7050, NORWAY\\
{\it E-mail address:} chao\_{}wang\_{}1987@126.com

\noindent Shicheng Wang, School of Mathematical Sciences, Peking University, Beijing 100871, CHINA\\
{\it E-mail address:} wangsc@math.pku.edu.cn

\noindent Yimu Zhang, Mathematics School, Jilin University, Changchun 130012, CHINA\\
{\it E-mail address:} zym534685421@126.com

\noindent Bruno Zimmermann, Universita degli Studi di Trieste, Trieste 34100, ITALY\\
{\it E-mail address:} zimmer@units.it

\end{document}